\documentclass{article}
\usepackage{amsmath,amsfonts,amssymb,amsthm,epsfig,epstopdf,titling,url,array}
\usepackage{eqnarray}
\usepackage{enumerate}
\usepackage[a4paper]{geometry}
\usepackage{amscd}
\usepackage{centernot}
\usepackage{authblk}
\theoremstyle{plain} 
\newtheorem{thm}{ Theorem}[section]
\newtheorem{lem}[thm]{Lemma}
\newtheorem{prop}[thm]{Proposition}
\newtheorem{cor}[thm]{Corollary}
\newtheorem{defn}[thm]{Definition}
\newtheorem{ex}[thm]{Example}

\newtheorem* {note*}{Note}

\newtheorem{qu}{\bf Question}
\newtheorem{con}{\bf Conjecture}
\newcommand{\co} {\mathbb{C}}

\newcommand{\iy} {\infty}
\newcommand{\N} {\mathbb{N}}

\newcommand{\Z} {\mathbb{Z}}

\newcommand{\lm}{\lambda}

\newcommand{\h}{\mathcal H}

\newcommand{\B}{\mathcal B}
\newcommand{\BH}{\mathcal B(\mathcal H)}

\newcommand{\norm}[1]{\left\Vert#1\right\Vert}

\newcommand{\set}[1]{\left\{#1\right\}}

\newcommand{\brc}[1]{\left(#1\right)}

\newcommand{\LZ}{\ell^{2}(\mathbb Z)}
\newcommand{\LN}{\ell^{2}(\mathbb N)}
\newcommand{\LNp}{\ell^{p}(\mathbb N)}
\newcommand{\m}{\mathcal {M}}

\title{\bf \Large Some properties of subspaces-hypercyclic operators}
\bf
\author[1]{\bf\footnotesize Nareen Sabih \thanks{nareen\_bamerni@yahoo.com}}
\author[2]{\bf Adem K{\i}l{\i}\c{c}man \thanks{akilicman@yahoo.com}}
\affil[1,2]{\bf Department of Mathematics, University Putra Malaysia,
43400 UPM, Serdang, Selangor, Malaysia}

\begin{document}
\date{}
\maketitle
\begin{abstract}
In this paper, we answer a question posed in the introduction of \cite{sub hyp} positively, i.e, we show that if $T$ is $\m$-hypercyclic operator with $\m$-hypercyclic vector $x$ in a Hilbert space $\h$, then $P(Orb(T,x))$ is dense in the subspace $\m$ where $P$ is the orthogonal projection onto $\m$. Furthermore, we give some relations between  $\m^{\perp}$-hypercyclicity and the orthogonal projection onto  $\m^{\perp}$.  We also give sufficient conditions for a bilateral weighted shift operators on a Hilbert space $\LZ$ to be subspace-hypercyclic, cosequently, there exists an operator $T$ such that both $T$ and $T^*$ are subspace-hypercyclic operators. Finally, we give an $\m$-hypercyclic criterion for an operator $T$ in terms of its eigenvalues.
\end{abstract}
\section{Introduction}
A bounded linear operator $T$ on a separable Hilbert space $\h$ is hypercyclic if there is a vector $x\in \h$ such that $Orb(T,x)=\set{T^nx:n\ge 0}$ is dense in $\h$, such a vector $x$ is called  hypercyclic for $T$. The first example of a hypercyclic operator on a Banach space was constructed by Rolewicz in 1969 \cite{Rolewicz}. He showed that if $B$ 
is the backward shift on $\LNp$ then $\lm B$ is hypercyclic if and only if $|\lm|> 1$. The hypercyclicity concept was probabely born with the thesis of Kitai in 1982 \cite{Kitai} who introduced the hypercyclic criterion to ensure the existence of hypercyclic operators. For more information on hypercyclic operators we refer the reader to \cite{dynamic, Erdman} \\

In 2011,  Madore and Mart\'{i}nez-Avenda\~{n}o \cite{sub hyp} studied the density of the orbit in a non-trivial subspace instead of the whole space, this phenomenon is called   the subspace-hypercyclicity. For the series of references on subspaces-hypercyclic operators see \cite{Some ques, C.M, sub hyp, notes on sub}

\begin{defn}\cite{sub hyp}
Let $T\in \BH$ and $\m$ be a closed subspace of $\h$. Then $T$ is called $\m$-hypercyclic or subspace-hypercyclic operator for a subspace $\m$  if there exists a vector $x\in \h$ such that $Orb(T,x)\cap \m $ is dense in $\m$. We call $x$ an $\m$-hypercyclic vector for $T$.
\end{defn}

\begin{defn}\cite{sub hyp}
Let $T \in \BH$ and $\m$ be a closed subspace of $\h$. Then $T$ is called $\m$-transitive or subspace-transitive for a subspace $\m$ if for each pair of non-empty open sets $U_1,\,U_2$ of $\m$ there exists an $n \in \N$ such that $T^{-n}U_1 \cap U_2$ contains a non-empty relatively open set in $\m$.
\end{defn}

\begin{thm}\cite{sub hyp}\label{MTD}
Every $\m$-transitive operator on $\h$ is $\m$-hypercyclic. 
\end{thm}

\begin{thm}\cite{m2}\label{hm}
Every hypercyclic operator is subspace-hypercyclic for a subspace $\m$ of $\h$.
\end{thm}

In the present paper, we answer positively a question posed by Madore and Mart\'{i}nez-Avenda\~{n}o in the introduction of \cite {sub hyp}. In particular, Theorem \ref{proj} shows that if $P$ is the orthogonal projection onto a subspace $\m$, then $P(Orb(T,x)$ will be dense in $\m$ whenever $T$ is $\m$-hypercyclic. Furthermore, we give some relations between $\m^{\perp}$-hypercyclicity and the orthogonal projection onto  $\m^{\perp}$. \\
Through Theorem \ref{5}, we give a set of sufficient conditions for a bilateral weighted shift on a Hilbert space $\LZ$ to be $\m$-hypercyclic operator. Like the hypercyclicity case, Example \ref{adjoint} shows that there is an operator $T$ such that both $T$ and its adjoint are subspace-hypercyclic; however, we dont know whether they are subspace-hypercyclic for the same subspace or not.\\ 
According to which an operator having a large supply of eigenvectors is hypercyclic, Godefroy-Shapiro \cite{28} exhibited a hypercyclic criterion which is called ``Godefroy-Shapiro Criterion". We extend such a criterion to a subspaces and we call it ``Spectrum $\m$-hypercyclic Criterion". We give the Example \ref{3} to show that the $\m$-hypercyclic criterion is a stronger result than Spectrum $\m$-Hypercyclic Criterion. \\

\section{Main results}

The set of all $\m$-hypercyclic operators is denoted by $HC(\h,\m)$ and the set of all $\m$-hypercyclic vectors of $T$ is denoted by $HC(T,\m)$. \\
To prove the following results, the reader should be convenient with the properties of the projection map $P$, see \cite{conway}.\\
In the introduction of \cite{sub hyp} , the authors asks the possibility of $P(Orb(T,x))$ to be dense in $\m$ where $P$ is the orthogonal projection onto $\m$. 
\begin{thm}\label{proj}
If $x\in HC(T,\m)$ and $P:\h \rightarrow \m$ is the orthogonal projection onto $\m$, then $P\brc{Orb(T,x)}\cap \m$ is dense in $\m$.
\end{thm}

\begin{proof}
Since $x\in HC(T,\m)$, then there exist a sequence $\set{n_k}$ of positive numbers such that $Orb(T,x) \cap \m= \set{T^{n_k}x:k\in \N}$ is dense in $\m$. Since $$\set{T^{n_k}x:k\in \N} \subseteq P\brc{Orb(T,x)} \cap \m,$$ then $P\brc{Orb(T,x)}\cap \m$ is dense in $\m$.
\end{proof}
The following examples gives an application to Theorem \ref{proj}

\begin{ex}
Let $B$ be the backward shift on $\LN$ and $\m$ be the subspace of $\LN$ consisting of all sequences with zeroes on the even entries; that is
$$\m=\set{\set{x_n}_{n=0}^\iy \in \LN: x_{2k}=0 \mbox { for all } k}$$
Since $2B\in HC(\m,\h)$ see \cite[Example 3.8.]{sub hyp}, then there exists an $x\in HC(2B,\m)$. We may, and will, assume that $x=\brc{a_0,0,a_1,0,\ldots} \in \m$, then for all $n\ge 0$, 

$$(2B)^{2n+1}x=2^{2n+1}\brc{0,a_{2n+2},0,a_{2n+4},\ldots}$$ 
and  
$$(2B)^{2n}x=2^{2n}\brc{a_{2n},0,a_{2n+2},0,a_{2n+4},\ldots}.$$

Therefore 
\begin{equation}\label{1}
Orb(2B,x) \cap \m=(2B)^{2n}x \mbox{ is dense in } \m 
\end{equation}
Moreover

\begin{equation*}
P\brc{Orb(T,x)}=
\begin{cases} 0 & \text{if $n$ is odd,}
\\
(2B)^{n}x &\text{if $n$ is even.}
\end{cases}
\end{equation*}
Thus, it is clear from  eq. (\ref{1}) that $(2B)^{2n}x \subset P\brc{Orb(T,x)}\cap \m=(2B)^{2n}x \cup \set{0}$ for all $n\ge 0$ and this gives the proof.
\end{ex}
\begin{prop}
Let $\m^{\perp}$ be an invariant subspace under $T$,  $P$ be the orthogonal projection onto $\m^{\perp}$ and $T\in HC(\h,\m^{\perp})$, then 
\begin{enumerate}
\item $S\in HC(\h,\m^{\perp})$ where $S=PT:\m^{\perp} \rightarrow \m^{\perp}$,
\item $\bar{T}\in HC(\h/\m)$ where $\bar{T}:\h/\m \rightarrow \h/\m$.
\end{enumerate}
\end{prop}
\begin{proof}
{\bf (1):} Let $x\in HC(T,\m^{\perp})$ where $x=u+v$; $u\in \m$ and $v\in \m^{\perp}$, then $Px=v$. Without loss of generality, we may assume that $x\in \m^{\perp}$, then 
\begin{eqnarray*}
Orb(PT,v)\cap \m^{\perp}&=&\set{v,PTv,(PT)^2v,\ldots}\cap \m^{\perp}\\
&=&\set{v,PTv,PT^2v,\ldots}\cap \m^{\perp}\\
&=&\set{Px,PTPx,PT^2Px,\ldots}\cap \m^{\perp}\\
&=&\set{Px,PTx,PT^2x,\ldots}\cap \m^{\perp}\\
&=& P\brc{ Orb(T,x)}\cap \m^{\perp}\\
&=& P\brc{Orb(T,x)}
\end{eqnarray*}
Since $P$ is continuous, then
\begin{eqnarray*}
\overline{\brc{ Orb(PT,v)\cap \m^{\perp}}}&=& \overline{P\brc{ Orb(T,x)}}\\
&\supseteq& P\brc{\overline{ Orb(T,x)}}\\
&\supseteq& \m^{\perp}.
\end{eqnarray*}
Thus $PT\in HC(\h,\m^{\perp})$.\\

{\bf (2):}
Without loss of generality, we may assume that $x\in \m^{\perp}$. By part (1), $PT\in HC(\h,\m^{\perp})$ where $PT:\m^{\perp} \rightarrow \m^{\perp}$ and since $(PT)^nx = PT^nx$ for all $n\ge 0$, then $P\brc{ Orb(T,x)}$ is dense in $\m^{\perp}$. Now, we will show that $x+\m \in HC(\bar{T})$,
\begin{eqnarray*}
\overline{\set{(\bar{T}^n(x+\m)): n\ge 0}}&=&\overline{\set{ (T^nx+\m): n\ge 0}}\\
&=&\overline{\set{PT^nx:n\ge 0}}+\m, \mbox{ since } T^n x\in  \m^{\perp}\\
&=&\overline{P( Orb(T,x)}+\m\\
&=&\m^{\perp}/\m=\h/\m.
\end{eqnarray*}
\end{proof}
The next Theorem gives sufficient conditions for a bilateral weighted shift operator on the Hilbert space $\LZ$ to satisfy the $\m$-hypercyclic criterion. We will suppose that 
$$\B=\set{e_{m_r}: r\in \N}$$ 
is a Schauder basis for $\m$, where $m_r\in \Z$. Let $T$ be the bilateral forward weighted shift operator with a weigh sequence $w_n$, then  
$T(e_{m_r})=w_{m_r}e_{m_r+1}$ for all $r\in \N$. We may define a right inverse (backward shift) $S$  to $T$ as follows: $S(e_{m_r})=\frac{1}{w_{m_r-1}}e_{m_r-1}$. Observe that $TSe_{m_r}=e_{m_r}$ for all $r\in \N$. If $T$ is invertable then $T^{-1}=S$. Also note that for all $r\in \N$ and $k\ge 0$, we have

$$\displaystyle T^k(e_{m_r})=\brc{\prod_{j=m_r}^{m_r+k-1}w_j}e_{m_r+k} \mbox{	and } \displaystyle S^k(e_{m_r})=\brc{\prod_{j=m_r-1}^{m_r-k}\frac{1}{w_{j}}}e_{m_r-k}$$
First we need the following lemma.

\begin{lem}\label{5}
Suppose that $T$ is an invertable bilateral weighted shift, $\set{n_k}$ is a sequence of positive integers, $n_k \to \iy$ and $\m$ is closed subspace such that $T^{n_k}\m \subseteq \m$. If there exists an $i\in \N$ such that $T^{n_k}e_{m_i}\to 0$ as $k\to \iy$, then $T^{n_k}e_{m_r}\to 0$ for all $r\in \N$. 
\end{lem}
\begin{proof}
Since $T^{n_k}\m \subseteq \m$, the proof is similar to the proof of \cite[Lemma 3.1.]{H C Bi.}.
\end{proof}
\begin{thm} \label{forw}
Let $T$ be an invertable bilateral forward weighted shift in the Hilbert space $\LZ$ with a positive weight sequence
$\set{w_n}_{n\in\Z}$ and  $\m$ be a subspace of $\LZ$, then $T$ satisfies the $\m$-hypercyclic criterion if there exists a sequence of positive integers $\set{n_k}$ such that $T^{n_k}\m\subseteq \m$ and $i\in \N$ such that
  	$$\displaystyle\lim_{k\to \iy}\prod_{j=m_i}^{m_i+n_k-1}w_j=0 \mbox{	and } \displaystyle\lim_{k\to \iy}\prod_{j=m_i-1}^{m_i-n_k}\frac{1}{w_{j}}=0$$
\end{thm}

\begin{proof}
We will verify the $\m$-hypercyclic criterion with $D=D_1=D_2=$ be a dense subset of $\m$ consisting of all finite sequences, that is, those sequences only have finite number of nonzero entries. Let $x\in D$, then it suffices to show that $x=e_{m_r}$ for $r\in \N$. Furthermore, by Lemma \ref{5} it suffices to show that $T^{n_k}e_{m_i} \to 0$ and $S^{n_k}e_{m_i} \to 0$. But that is clear because $\displaystyle \norm{T^{n_k}e_{m_i}}=\prod_{j=m_i}^{m_i+n_k-1}w_j\to 0$ and $\displaystyle \norm{S^{n_k}e_{m_i}}=\prod_{j=m_i-1}^{m_i-n_k}\frac{1}{w_{j}}\to 0$. Moreover it is clear that $T^{n_k}S^{n_k}x=x$. By taking $x_k=S^{n_k}$, it is clear that the conditions of $\m$-hypercyclic criterion are satisfied.
\end{proof}
By the same way we can characterize the $\m$-hypercyclic backward weighted shifts since they are unitarily equivalent to forward shifts.
\begin{cor}\label{bac}
If $S$ is an invertable bilateral backward weighted shift in the Hilbert space $\LZ$ with a positive weight sequence
$\set{w_n}_{n\in\Z}$ and  $\m$ is a subspace of $\LZ$, then $T$ satisfies the $\m$-hypercyclic criterion if there exists a sequence of positive integers $\set{n_k}$ such that $T^{n_k}\m\subseteq \m$ and $i\in \N$ such that
  	$$\displaystyle\lim_{k\to \iy}\prod_{j=m_i}^{m_i-n_k+1}w_j=0 \mbox{	and } \displaystyle\lim_{k\to \iy}\prod_{j=m_i+1}^{m_i+n_k}\frac{1}{w_{j}}=0$$
\end{cor}
From the above theorem, one can construct a weight sequence $\set{w_n}$ such that $\set{w_n}$ satisfy the conditions of Theorem \ref{forw} and Corollary \ref{bac}. It immediately follows that both $T$ and $T^*$ are subspaces-hypercyclic operators.  In other words, we can show that the Herrero question \cite{limits} holds true even on a subspace of a Hilbert space.  

\begin{ex}\label{adjoint} 
Let $T$ be a bilateral shift defined as in the example of  Salas \cite{3}. Then, both $T$ and $T^*$ are $\m$-hypercyclic operators.
\end{ex}
\begin{proof}
The proof follows directly from theorem \ref{hm}. In particular,  there exist two subspaces $\m_1$ and $\m_2$ such that $T$ is $\m_1$-hypercyclic and $T^*$ is $\m_2$-hypercyclic.
\end{proof}
\begin{qu}
If $T$ is $\m_1$-hypercyclic and $T^*$ is $\m_2$-hypercyclic. What is the relation between $\m_1$ and $\m_2$?
\end{qu}
\begin{con}
If $T$ is $\m_1$-transitive and $T^*$ is $\m_2$-transitive, then $\m_2=\m_1^\perp$.
\end{con}
\begin{thm}[Spectrum $\m$-Hypercyclic Criterion]\label{smc}
Let $T\in\BH$ and $\m$ be a subspace of $\h$. If there is a positive integers $p$ such that the subspace\\
$X = span\set{x \in \h ; T^px = \lm x, \lm \in \co, |\lm| < 1}\cap \m,$\\
$Y = span\set{x \in \h ; T^px = \lm x, \lm \in \co, |\lm| > 1}\cap \m$\\
are dense in $\m$. Then $T$ is $\m$-transitive. 
\end{thm}
\begin{proof}
Let $U$, $V$ be nonempty open subsets of $\m$. By hypothesis there exist $x \in X \cap U$ and $y \in Y \cap V$ such that 

$$x=\sum_{k=1}^m a_kx_k \quad{} and \quad{}  y=\sum_{k=1}^m b_ky_k$$
where $x_k \in X$ and $y_k \in Y$ which mean that $T^px_k=\lm_{k}x_k,\, |\lm_k|<1$ and $T^py_k=\mu_{k}y_k,\, |\mu_k|>1$, for some $a_k,b_k,\lm_k,\mu_k \in \co,\, k=1,\ldots, m$. Let 
$z_n=\sum_{k=1}^m b_k \frac{1}{\mu_k^n}y_k$. Since as $(n)\to \iy$,
$$T^{n+p}x=\sum_{k=1}^m a_k\lm_k^{n}x_k\to 0 \mbox { and } z_{n} \to 0$$
 and since $T^{n+p}z_{n}=y$ for all $n\ge 0$, then there exists $N\in \N$ such that, for all $n\ge N$,
$$x+z_{n}\in U \quad{} and \quad{} T^{n+p}(x+z_{n})=T^{n+p}x+y\in V.$$

It follows that for all ${n}\ge N$,
$$T^{n+p}U \cap V\neq \phi $$

Now, let $w\in \m$, since $X$ is dense in $\m$, then there exists a sequence $\brc{x_k}\in X$ such that $\brc{x_k}\to w$. since $T\in\BH$, then $T^{n+p}(x_k)=(\lm_k^{n}x_k)\to T^{n+p}w$ for all ${n}\ge 0$. Since $|\lm_k|<1$ for all $k \ge 0$, then  $(\lm_k^{n}x_k)\to 0$ as ${n}\to \iy$. Thus, $0=T^{n+p}w \in \m$ and we get $T^{n+p}(\m)\subseteq \m$ for all ${n}\ge N$.
This show that $T$ is $\m$-transitive.
\end{proof}
We give an example showing that the $\m$-hypercyclic criterion is a stronger result than Spectrum $\m$-Hypercyclic Criterion.
\begin{ex}\label{3}
Consider the bilateral weighted forward shift $T=F:\LZ \to \LZ$, with the weight sequence 
\begin{equation*}
w_n=
\begin{cases} \frac{1}{2} & \text{if $n \geq 0$,}
\\
3 &\text{if $n< 0$.}
\end{cases}
\end{equation*}
Let $\m$ be the subspace of $\LZ$ consisting of all sequences with zeroes on the even entries; that is,
 $$\m=\left\{\left\{a_n\right\}_{n=-\iy}^{\iy}\in \LZ : a_{2n}=0, n \in \Z\right\}$$
then $T$ satisfies $\m$-hypercyclic criterion, but not spectrum $\m$-hypercyclic criterion.
\end{ex}
\begin{proof}
Applying Theorem \ref{forw} with $n_k=2k$ and $m_i=1$, we get 
  	$$\displaystyle\lim_{k\to \iy}\prod_{j=1}^{2k}w_j=\lim_{k\to \iy}\prod_{j=1}^{2k}\frac{1}{2}=0 \mbox{	and } \displaystyle\lim_{k\to \iy}\prod_{j=-1}^{-2k}\frac{1}{w_{j}}=\lim_{k\to \iy}\prod_{j=1}^{2k}\frac{1}{3}=0$$

It can be easily deduced from the definition of $\m$ that  for each sequence $x \in \m$ and each $k \in \Z$, the sequence $T^{2k}x$ will have a zero entery on all even positions. It follows that  $T^{n_k}\m\subseteq \m$.

Thus $T$ satisfies the $\m$-hypercyclic criterion.

On the other hand, let $x=(\ldots,0,x_{-3},0,x_{-1},0,x_1,0,x_3,0, \ldots)$ be a non-zero element in $\m$. Towards a contradiction, assume that there exist $p\in \N$ such that  $T^px = \lm x$, then if $p$ is odd number, then we get $x=0$; otherwise, without loss of generality we will assume that $p=2$.  The equality $T^2x = \lm x$ implies that
$$x=\brc{\ldots,\brc{\frac{\lm}{3^2}}^2x_{-1},0,\brc{\frac{\lm}{3^2}}x_{-1},0,x_{-1},0, 3\frac{1}{2\lm}x_{-1},0,3\frac{1}{2^3\lm^2}x_{-1},0,3\frac{1}{2^5\lm^3} x_{-1},0, \ldots}$$

$$=\brc{\ldots,\brc{\frac{\lm}{3^2}}^2x_{-1},0,\brc{\frac{\lm}{3^2}}x_{-1},0,x_{-1},0, 6\brc{\frac{1}{2^2\lm}}x_{-1},0,6\brc{\frac{1}{2^2\lm}}^2x_{-1},0,6\brc{\frac{1}{2^2\lm}}^3 x_{-1},0, \ldots}$$
where $\lm\neq 0$ and $x_{-1}\neq 0$. But then we have that
$$\norm{x}=|x_{-1}|\brc{1+\sum_{n\in \N}\brc{\frac{|\lm|}{3^2}}^{n}+6\sum_{n\in \N}\brc{\frac{1}{2^2|\lm|}}^{n}}=\iy$$
whatever the value of $\lm$. Thus $T^2$ has no eigenvalues and therefore it does not satisfy the spectrum $\m$-hypercyclic criterion.
 \end{proof}


\end{document}